\documentclass[11pt]{amsart}
\usepackage[margin=1.2in,marginparsep=0.1in,marginparwidth=1in]{geometry}
\usepackage{amssymb,amsmath,amsthm,amstext,amscd,latexsym,graphics,graphicx,bbm,caption}
\usepackage[usenames,dvipsnames,svgnames,table]{xcolor}
\usepackage[plainpages=false,colorlinks=true, pagebackref]{hyperref}

\usepackage{mathtools}
\usepackage{tikz-cd}
\usepackage{cleveref}
\usepackage{enumitem}
\usepackage{setspace}
\hypersetup{citecolor=Sepia,linkcolor=blue, urlcolor=blue}

\usepackage{cite}  

\makeatletter
\def\@tocline#1#2#3#4#5#6#7{\relax
	\ifnum #1>\c@tocdepth 
	\else
	\par \addpenalty\@secpenalty\addvspace{#2}%
	\begingroup \hyphenpenalty\@M
	\@ifempty{#4}{%
		\@tempdima\csname r@tocindent\number#1\endcsname\relax
	}{%
		\@tempdima#4\relax
	}%
	\parindent\z@ \leftskip#3\relax \advance\leftskip\@tempdima\relax
	\rightskip\@pnumwidth plus4em \parfillskip-\@pnumwidth
	#5\leavevmode\hskip-\@tempdima
	\ifcase #1
	\or\or \hskip 2em \or \hskip 2em \else \hskip 3em \fi%
	#6\nobreak\relax
	\dotfill\hbox to\@pnumwidth{\@tocpagenum{#7}}\par
	\nobreak
	\endgroup
	\fi}
\makeatother
\newtheorem{intro-thm}{Theorem}[]
\theoremstyle{plain}
\newtheorem{thm}{Theorem}[section]
\newtheorem{theorem}[thm]{Theorem}
\newtheorem{question}[thm]{Question}
\newtheorem{lemma}[thm]{Lemma}
\newtheorem{corollary}[thm]{Corollary}
\newtheorem{proposition}[thm]{Proposition}

\theoremstyle{definition}
\newtheorem{remark}[thm]{Remark}
\newtheorem{construction}[thm]{Construction}

\newtheorem{definition}[thm]{Definition}

\newcommand{\Hom}{{\rm Hom}}

\newcommand{\im}{{\rm Im}}
\newcommand{\Res}{{\rm Res}}
\newcommand{\inv}{{\rm inv}}

\newcommand{\A}{{\mathbb A}}

\newcommand{\G}{{\mathbb G}}

\renewcommand{\P}{{\mathbb P}}
\newcommand{\Q}{{\mathbb Q}}

\newcommand{\Z}{{\mathbb Z}}

\input{xy}

\newcommand{\et}{\text{et}}
\DeclareMathOperator{\Br}{Br}

\begin{document}
	\title{
    Pullback Method with Applications to Severi--Brauer Fibrations}
	\author[M. Biswas]{Mridul Biswas} \address{School of Mathematics and Statistics, University of Canterbury, Christchurch, 8041, New Zealand } \email{biswasmridul123@gmail.com}
	\author[D. C. R]{Divyasree C Ramachandran} \address{Department of Mathematics, Indian Institute of Science Education and Research Pune, Dr Homi Bhabha Rd, Pashan, Pune, 411008, India}\email{crdivya99@gmail.com}
	\author[B. Samanta]{Biswanath Samanta} \address{Department of Mathematics, Indian Institute of Science Education and Research Pune, Dr Homi Bhabha Rd, Pashan, Pune, 411008, India}\email{biswa.uoh@gmail.com}
	
	
	\begin{abstract}
Given a variety with a suitable Brauer class, we present a general pullback construction that produces varieties that has Brauer--Manin obstruction to the existence of rational points. We then study Severi--Brauer fibrations and their Brauer groups without relying on explicit defining equations. As a key application, we show that there exist Severi--Brauer fibrations with index one that fails Hasse principle.
	\end{abstract}
	
	\maketitle

\section{Introduction}
\subsection{Background}
Let $X$ be a smooth projective geometrically integral variety defined over a global field $k$. A central question in arithmetic geometry is whether $X$ has a $k$-point. A necessary condition for the existence of such a point is that $X$ has $k_{v}$-points for all places $v$ of $k$. We know that $X(k_{v}) \neq \emptyset$ for almost all places $v$ of $k$ \cite[Theorem 7.7.2]{poonen}, which allows us to check whether \( X(\A_k) \neq \emptyset \) in a finite process.  However, this condition is not sufficient for many varieties; it is possible for \( X(k_v) \neq \emptyset \) for all places \( v \) of \( k \) while \( X(k) = \emptyset \). This is referred to as the failure of the Hasse principle.
    
	In 1970, Manin introduced a construction that explains the failure of the Hasse principle for all the varieties known at that time, which is now known as the Brauer--Manin obstruction. By a `nice' variety we mean a smooth projective, geometrically integral variety. Although there are counterexamples to the Hasse principle that cannot be captured by the Brauer--Manin obstruction \cite{sko}, Colliot-Th\'{e}l\`ene conjectured that for `nice' rationally connected varieties, the Brauer--Manin obstruction is the only obstruction to the Hasse principle.  When it is known to be the only obstruction, this gives a finite procedure for checking the existence of rational points. 
	
    By the functoriality of the Brauer group, we have a pairing \(\langle \cdot, \cdot \rangle : X(\mathbb{A}_k) \times \Br X \to \mathbb{Q}/\mathbb{Z}\). The Brauer--Manin set is defined as the set of all adelic points orthogonal to $\Br X$ and is denoted by \( X(\A_k)^{\Br} \). We have $X(k) \subset X(\A_k)^{\Br} \subset X(\A_k)$. If $X(\A_k)$ is nonempty and $X(\A_k)^{\Br}$ is empty, we say that $X$ has Brauer--Manin obstruction to the existence of rational points. If $X(\A_k)^{\Br}$ is a proper subset of $X(\A_k)$ then we say $X$ has Brauer--Manin obstruction to weak approximation.

In the literature, most examples of computing the Brauer--Manin obstruction rely on explicit models. By contrast, the following pullback theorem depends only on the existence of a Brauer class satisfying suitable conditions, 
making it applicable even to varieties without explicit defining equations. The result is inspired by \cite[Theorem~7.1]{berg}, which treated the case of conic bundles, while the pullback technique itself originates in \cite{cp}. Building on these ideas, we extend the method to a broader class of varieties under suitable hypotheses. As a key application, we construct Severi--Brauer fibrations of index one that fail 
the Hasse principle.

\begin{theorem} \label{pullback}
Let \(X\) be a 
variety over a global field \(k\), equipped with a 
dominant morphism \(\pi : X \to \mathbb{P}^1\) whose generic fiber is 
geometrically irreducible. Suppose there exists 
\(\alpha \in \Br X\) giving a Brauer--Manin obstruction to weak approximation on \(X\), and a rational point \(P \in X(k)\) such that 
\(\alpha(P) = 0 \in \operatorname{Br} k\). Then there exists a nonconstant polynomial \(g \in k[t]\) such that the 
base change of \(X\) by \(g\) has Brauer--Manin obstruction to the 
existence of rational points.
\end{theorem}

\begin{remark}
If \(X(k) = \emptyset\), then for any base change \(\psi : \mathbb{P}^1_k \to \mathbb{P}^1_k\) we still have \(X_\psi(k) = \emptyset\), so there is no question of a Brauer--Manin obstruction. This explains the assumption \(X(k) \neq \emptyset\) in the theorem. Moreover, the condition that there exists a point \(P \in X(k)\) with \(\alpha(P) = 0\) can be replaced by requiring that \(0\) lies in the image of the evaluation of \(\alpha\) at each place of \(k\). Finally, one may also pull back \(X\) in such a way that the resulting variety has no Brauer--Manin obstruction to weak approximation.
\end{remark}

 For proper scheme $X$, due to the compactness of $X(\A_k)$, if \( X(\A_k)^{\Br} = \emptyset \), then this is captured by a finite subgroup of $\Br X$. We also study the exponent of finite group capturing the Brauer--Manin obstruction. Since a single Brauer class captures the obstruction in most cases, we investigate whether there is a bound on the order of that class.

\begin{definition}
  We say `a finite group $G$ captures Brauer--Manin Obstruction' if there exists a `nice' scheme $X$ for which $X(\A)^G=\emptyset$ but $X(\A)^H \neq\emptyset$ for all proper subgroups $H$ of $G$.
\end{definition}

 Most of the examples of Brauer--Manin obstruction in the literature are captured by a cyclic subgroup of $\Br X$. Recently, Jennifer Berg and their collaborators in an inspiring paper \cite{berg}, have shown that the subgroup may require arbitrarily many generators. In particular, for any given natural number $n$, $(\Z/2\Z)^n$ captures the obstruction. This motivates the following question.

\begin{question}\label{mainquestion}
    Which finite abelian groups can capture the Brauer--Manin Obstruction? 
\end{question}

To address this question, we study the Brauer--Manin obstruction in the setting of Severi--Brauer fibrations, as an application of Theorem~\ref{pullback}. It is well known that computing the Brauer group of a general Severi--Brauer fibration is a difficult problem. Since the Brauer group is a birational invariant of smooth projective varieties, a natural approach is to attempt its computation using only the generic fiber. The present work represents a step in this direction.

The index of the variety is the greatest common divisor of all possible degrees of points. For a $k$-variety \(X\), it is clear that if there exists a $k$-rational point on the variety, then the index equals 1. The converse fails in general, and we exhibit such a phenomenon in Theorem \ref{main}.
   It is a general philosophy that Brauer--Manin obstruction to the Hasse principle for rational points on a variety may exist even when there is no such obstruction for the existence of a zero-cycle of degree one. For Severi--Brauer fibrations, however, the Brauer--Manin obstruction is the only obstruction to the existence of zero-cycles of degree one. In particular, for these varieties, the absence of such an obstruction is equivalent to having index one. In the following, we provide examples illustrating this philosophy.
	
\begin{theorem}\label{main}
Given an odd prime $p$ and a number field $k$ which contains a primitive $p$th root of unity, there exists a Severi--Brauer fibration over $k$ with index one and has Brauer--Manin obstruction to the existence of rational points captured by a $p$-torsion Brauer class. 
\end{theorem}

As a corollary, we obtain a partial answer to Question~\ref{mainquestion}.
    
    \begin{corollary}\label{sub of main}
         For any odd prime $p$, the group $\Z/p\Z$ can capture the Brauer--Manin obstruction. Therefore, the exponent of a finite group capturing Brauer--Manin obstruction has no upper bound.
	\end{corollary}
    
Recently, Liang and Liu have announced in the preprint \cite{ll} that any finite abelian group can capture the Brauer--Manin obstruction.


\subsection{Overview of the paper}
We begin Section~\ref{prelim} by recalling key concepts on Severi--Brauer varieties and the Brauer--Manin obstruction. In Section~\ref{pull back method}, we prove the pullback theorem for arbitrary varieties, with applications given in Section~\ref{applications}. In Section \ref{Brauer Group of Severi--Brauer fibration}, for any choice of a cyclic algebra over $k(t)$, we construct a Severi--Brauer fibration. Then we ensure that the Brauer group of such a variety is a non-trivial $p$-torsion group by base changing to a Henselian local ring and using the purity theorem.   

In Subsection~\ref{subgroups capturing BMO}, for each odd prime \(p\), we consider the cyclic \(p\)-algebra \((a,f)\), with \(a\) and \(f\) subject to suitable technical conditions, and the associated Severi--Brauer fibration. Using the pullback Theorem~\ref{pullback}, we exhibit a local-global failure captured by an element of order \(p\), thereby proving Corollary~\ref{sub of main}. Finally, in Subsection \ref{index one failing Hasse}, we construct a Severi--Brauer fibration using \Cref{SB fibration violating Hasse} which admits a zero-cycle of degree one with Brauer--Manin obstruction to rational points.

\subsection{Notations} Throughout the article, a variety is a separated scheme of finite type over a field. By a `nice' scheme, we mean a smooth projective geometrically integral scheme. For a number field $k$, let $\zeta_n$ denote a primitive $n$th root of unity. Let $\Omega_k$ denote the set of all places of $k$. For each place \(v\in\Omega_k\), we write \(k_v\) for the completion of \(k\) at \(v\). Let \(\mathbb{A}_k\) denote its ring of adeles. For each non-archimedean place \(v\), let \(\mathcal{O}_v\) denote the ring of integers of \(k_v\) and \(\mathfrak{m}_v\) its maximal ideal.

 For a variety $X$ over $k$, $X^s$ denotes the base change of $X$ to the separable closure of $k$. For any scheme \( X \), let \( X^{(1)} \) be the set of codimension 1 points. For \( t \in X \), we denote the residue field at \( t \) by \( \kappa(t) \). If \( X \) is integral, the function field of \( X \) is denoted by \( \kappa(X) \).

\section{Preliminaries} \label{prelim}

    \subsection{Brauer group of a field}(See \cite[Section 2.4]{gs} for details)
	For any field \( k \), let \(\Br k\) denote its Brauer group. $\Br k$ classifies finite central division algebras over $k$ upto similarity. Let $A$ be a central simple algebra over $k$. Then $A$ is said to be split over $k$ if the class of $A$ in $\Br k$ is zero. Equivalently, $A$ is split if and only if there exists an integer $m >0$ such that $A \simeq M_m(k)$ as $k$-algebras. \\
	Given a central simple algebra $A$ of degree $n$ over an arbitrary field $k$, there exists an associated projective $k$ variety $X$ of dimension $n - 1$ known as the Severi--Brauer variety associated with $A$.
	The existence of a $k$-point in $X$ is characterized as follows.
	
	\begin{lemma}\label{SB variety and rational point}
		Let $A$ be a central simple algebra of degree $n$ over a field $k$ and $X$ be the associated Severi--Brauer variety over $k$. The following are equivalent.
		\begin{enumerate}[label=\arabic*.]
			\item $X$ has $k$-rational point.
			\item $X$ is isomorphic to the projective space $\P^{n-1}_{k}$ over $k$.
			\item $A$ splits over $k$. 
		\end{enumerate}
		
	\end{lemma}
	
	\begin{proof}
    Follows from \cite[Theorem 5.1.3]{gs} and the base-point preserving bijection in \cite[p. 143]{gs}.
	\end{proof}
	
	Cyclic algebras constitute an important subclass of central simple algebras. These are generalization of quaternion algebras to arbitrary degree. When the base field $k$ contains a primitive $m$-th root of unity $\zeta_m$, a degree $m$ cyclic algebra has a nice presentation. Given any $a,b \in k^\times$, there exists a cyclic $m$-algebra denoted by $(a,b)_{\zeta_m} $, which can be described as
	\begin{center}
		$ \langle x,y : x^m = a, y^m = b, xy=\zeta_m yx\rangle$.
	\end{center}
	
	We now present a useful result concerning cyclic algebras, whose proof is given in \cite[Corollary 4.7.8]{gs}.
	
    \begin{lemma}\label{split algebra} Let $k$ be a field containing a primitive $m$th root of unity $\zeta_m$ and let $a,b \in k^\times$. Then the following statements are equivalent.
    \begin{enumerate}[label=\arabic*.]
		\item The cyclic algebra $(a,b)_{\zeta_m}$ is split over $k$.
		\item The element $b$ is a norm from the field extension $k(\sqrt[m]{a})$.
		\item The element $a$ is a norm from the field extension $k(\sqrt[m]{b})$.
    \end{enumerate}
    \end{lemma}

\begin{theorem}[Amitsur]\label{Amitsur}
    Let $X$ be a Severi--Brauer variety defined over a field $k$. Then the kernel of the restriction map $r_X : \Br k \rightarrow \Br \kappa(X)$ is a cyclic group generated by the class of $X$ in $\Br k$.
\end{theorem}

The definition of a Brauer group can be generalized to a scheme using \'{e}tale cohomology.

\subsection{Brauer group of a scheme}
For any scheme $X$ over a field $k$, define the Brauer group as $\Br X:=H^2_{\et}(X,\G_m)$, and its subgroup of constant classes as $\Br_0X:=\im(\Br k\rightarrow \Br X)$. We write $\overline{\Br}X$ for the quotient $\Br X/\Br_0 X$.

Let $L$ be a $k$-algebra and $x \in X(L)$. Then the morphism $\text{Spec}\, L \xrightarrow{x} X$ induces, by functoriality of the Brauer group, a homomorphism $\Br X \to \Br L$. For $A \in \Br X$, its image under this homomorphism is called the \emph{evaluation} of $A$ at $x$, and is denoted by $A(x)$ or $x^*A$. The inclusion of the generic point $i: \text{Spec}\,\kappa(X) \to X$ induces a homomorphism $i^*: \Br X \to \Br \kappa(X)$. Finally, the Brauer group is a \emph{birational invariant} for smooth projective varieties. For more details, see \cite[§ 6.6]{poonen}.

 The following theorem helps us decide when an element of $\Br\kappa(X)$ actually belongs to $\Br X$. For a reference, see \cite[Theorem 6.8.3]{poonen}. This is a version of the Grothendieck's Purity theorem for the Brauer group.

\begin{theorem}[Purity theorem]
    Let $X$ be a regular integral Noetherian scheme. Then the following sequence 
    \begin{equation*}\label{purity1}
    0\rightarrow\Br X\rightarrow\Br \kappa(X)\xrightarrow{\oplus_x \partial_x}\displaystyle\bigoplus_{x\in X^{(1)}}H^1(\kappa(x),\Q/\Z)
	\end{equation*}
	is exact.
\end{theorem}
In particular, if we choose $X$ to be $\P^1_k$ and a cyclic algebra $(a,g)$ of index $p$ in $\Br\kappa(\P^1_k)$, then for closed points of $t\in{\P^{1}_k}$, $\partial_t(a,g)=a^{v_t(g)}$ in $\kappa(t)^{\times}/\kappa(t)^{\times p}\subset H^1(\kappa(t),\Q/\Z)$.

\begin{theorem}
   Let $X$ be a regular integral scheme and $U\subset X$ be a dense open subscheme. Then we have an exact sequence
     \begin{equation*}\label{purity2}
    0\rightarrow\Br X\rightarrow\Br U\xrightarrow\displaystyle\bigoplus_{D}H^1(\kappa(D),\Q/\Z)
	\end{equation*}
    where $D$ ranges over the irreducible divisors of $X$ with support in $X\setminus U$ and $\kappa(D)$ denotes the residue field at the generic point of $D$.
\end{theorem}
    Now we recall some terminology related to Brauer--Manin obstruction.
\subsection{Brauer--Manin obstruction}
Let $X$ be a variety over a number field $k$. One can define the Brauer--Manin pairing \( \langle \cdot, \cdot \rangle : X(\mathbb{A}_k) \times \Br X \rightarrow \Q/\Z \) by summing the local pairings
\(\big((P_v), A\big) \mapsto \sum_v \operatorname{inv}_v \big(A(P_v)\big)\). For $A \in \Br X$, we have the following commutative diagram:
\[
\begin{tikzcd}
& X(k) \arrow[r, hook] \arrow[d,"A"] & X(\A_k) \arrow[d,"A"] & & \\
0 \arrow[r] & \Br k \arrow[r] & \bigoplus_{v} \Br k_{v} \arrow[r, "\sum \operatorname{inv}_{v}"] & \Q/\Z \arrow[r] & 0
\end{tikzcd}
\]
and we define the subset of adelic points \[X(\A_k)^A := \left\{ (x_v) \in X(\A_k) : \sum_v \operatorname{inv}_v A(x_v) = 0 \right\}.\]
For any subgroup $B \subset \Br X$, set \( X(\A_k)^B := \bigcap_{A \in B} X(\A_k)^A\). The above diagram shows that \(X(k)\subset X(\A_k)^{\Br} \subset X(\A_k)\). We say that there is a \emph{Brauer--Manin obstruction} to the existence of rational points on $X$ if $X(\A_k) \neq \emptyset$ but $X(\A_k)^{\Br} = \emptyset$.

 For a regular projective variety $X$, the continuity of Brauer--Manin pairing ensures that if $X$ has Brauer--Manin obstruction then there exists a finite subgroup $B \subset \Br X$ that captures the obstruction.
Let $B$ be such a finite subgroup, and write $\hat{B} := \operatorname{Hom}(B, \Q/\Z)$. For each place $v$, we obtain a map
\[
\phi_v : X(k_v) \to \hat{B},\] 
given by $P_v \mapsto \big(A \mapsto \operatorname{inv}_v A(P_v)\big).$
These maps sum to give a single map
\[
\phi : X(\A_k) \rightarrow \hat{B}.
\]
Moreover, for all but finitely many places $v$, one has $\phi_v\big(X(k_v)\big) = \{0\}$ \cite[Proposition ~13.3.1]{BG}, which ensures that the sum is well defined.

For each place $v$, the set $X(k_v)$ has a natural topology induced by that of $k_v$.  
For any subset $S \subset \Omega_k$, we equip $\prod_{v \in S} X(k_v)$ with the product topology. We say that a variety $X$ over $k$ satisfies \emph{weak approximation} if, for every finite subset $S \subset \Omega_k$, the image of the diagonal embedding
\[
X(k) \rightarrow \prod_{v \in S} X(k_v)
\]
is dense.

Let $\overline{X(k)}$ denote the closure of $X(k)$ in the adelic space $X(\mathbb{A}_k)$.  
By \cite[Corollary 8.2.11]{poonen}, one has $\overline{X(k)} \subset X(\mathbb{A}_k)^{\Br}$.

If $X$ is proper and $X(\mathbb{A}_k)^{\Br} \neq X(\mathbb{A}_k)$, then weak approximation fails for $X$. In this case, we say that there is a \emph{Brauer--Manin obstruction to weak approximation} on $X$. More precisely, for $\alpha \in \Br X$, let $\langle \alpha \rangle$ denote the subgroup of $\Br X$ generated by $\alpha$. If $X(\mathbb{A}_k)^{\langle \alpha \rangle} \neq X(\mathbb{A}_k)$, we say that $\alpha$ gives a Brauer--Manin obstruction to weak approximation on $X$. 

\section{Pull back method}\label{pull back method}

 In this section we show how from a given a variety with a global point we construct a variety that violates the Hasse principle by pulling back along a carefully chosen polynomial.
 By using the pullback method, we can modify our variety in a way that allows us to control the images of the Brauer--Manin pairing at finitely many places. This, in turn, enables us to control which subgroup of the Brauer group captures the obstruction.
	
\begin{proof}[{Proof of Theorem \ref{pullback}}]
     For a place \(v\) of \(k\), let $\phi_v: X(k_v) \to \Hom(\langle \alpha \rangle, \Q/\Z)$ denote the map induced by the Brauer--Manin pairing. By \cite[Prop 13.3.1]{BG}, the image of \(\phi_v\) is zero for almost all places $v$ of $k$. By assumption, there exists a point \(P \in X(k)\) such that \(\alpha(P)=0\). Since a \(k\)-rational point gives rise to a \(k_v\)-rational point for every place \(v\) of \(k\), it follows that \(0\) lies in the image of \(\phi_v\) for all \(v\). \\
     Let $v_i$, for  $1 \leq i \leq r$, denote the non-archimedean places  at which the map $\phi_v $ is non-constant. Since \(\alpha\) obstructs weak approximation \(X\), the set \(X(\mathbb{A}_k)^{\alpha}\) is a proper subset of \(X(\A_k)\). Thus there exists atleast one \(i\) with \(1 \leq i \leq r\) and a point \(P_{v_i} \in X(k_{v_i})\) such that \(\alpha(P_{v_i})\neq 0\). In particular, there exists a non-archimedean place, say \(v_1\), for which \(\im(\phi_{v_1}) \neq \{0\}\).
     
     Fix a non-zero element \(s \in \im(\phi_{v_1})\). Let $\mathcal{U}_{v_1}= \pi(\phi_{v_1}^{-1}(\{s\}))$ and for $2 \leq i \leq r$, let $\mathcal{U}_{v_i} := \pi(\phi_{v_i}^{-1}(\{0\}))$. Each \( \mathcal{U}_{v_i}\) is non-empty open subset of $\P^1(k_{v_i})$. Using weak approximation on \(\P^1\), we can choose a \( k \)-point $c$ in \(\mathbb{P}^1\) that lies in \( \mathcal{U}_{v_i} \) for every \( i \) and has geometrically irreducible fiber. We may assume that \( c = \infty \) after making an appropriate change of coordinates on \(\mathbb{P}^1\). This ensures that the fiber \( X_\infty(k_{v_i}) \neq \emptyset \) for all \( 1 \leq i \leq r \). 

     Now we consider the following conditions on a polynomial $g(t) \in k[t]$:
\begin{enumerate}[label=\arabic*.]
    
      \item For each place \(v \) such that \( {X}_{\infty}(k_v) = \emptyset \), the set \( g(\P^1(k_v)) \) meets \( {\pi}({X}(k_v)) \). 
    \item For each \(1 \leq i \leq r \), we have \( g(\P^1(k_{v_i})) \subset \mathcal{U}_{v_i} \).
    
\end{enumerate}
For each place \(v\) there exist local polynomials satisfying each of the above open conditions, by \cite[Lemma~4.3]{berg}. Hence, by weak approximation, we may choose a polynomial \( g \in k[t] \) that satisfies both of them simultaneously. Let \(\pi_g: X_g \longrightarrow \mathbb{P}^1\) denote the pullback of the morphism \( X \to \mathbb{P}^1 \) along \( g: \mathbb{P}^1 \to \mathbb{P}^1 \).\\
Let \( \tilde{g}: X_g \to X \) be the morphism induced by the polynomial \( g \), and set \(\gamma = \tilde{g}^*\alpha \in \Br X_g\) to be the pullback of \(\alpha\). For each place \( v \) of \( k \), let  
\(
\tilde{\phi}_v: X_g(k_v) \longrightarrow \Hom(\langle \gamma \rangle, \Q/\Z)
\)
denote the corresponding evaluation map for \( X_g \). Then the composition \(\phi_v \circ \tilde{g}\) factors through \(\tilde{\phi}_v\) and the natural inclusion  
\(\Hom(\langle \gamma \rangle, \Q/\Z) \hookrightarrow \Hom(\langle \alpha \rangle, \Q/\Z)\).\\
Condition (1) guarantees that $X_g(k_v) \neq \emptyset$ for every place $v$ where $X_\infty(k_v) = \emptyset$. Hence, we conclude that $X_g(\mathbb{A}_k) \neq \emptyset$. Condition (2) guarantees that the image of \( \tilde{\phi}_{v_1} \) is \(\{s\}\), and for each \( 2 \leq i \leq r \), the image of \( \tilde{\phi}_{v_i} \) is exactly \(\{0\}\). Moreover, for places \( v \notin \{v_1, \dots, v_r\} \), we have \(\im(\phi_{v}) = \{0\}\), which implies \(\im(\tilde{\phi}_{v}) = \{0\}\). Thus the only nontrivial image occurs at \( v_1 \). Consequently, \(0\) does not lie in the image of the map \( \tilde{\phi}: X_g(\mathbb{A}_k) \rightarrow \Hom(\langle \gamma \rangle, \Q/\Z)\) and we conclude that  
\(X_g(\mathbb{A}_k)^{\gamma} = \emptyset\). Hence, \( X_g \) admits a Brauer--Manin obstruction to the existence of rational points.

\end{proof}
\section{Brauer groups of Severi--Brauer fibrations}\label{Brauer Group of Severi--Brauer fibration}

In this section, we describe the construction of a Severi--Brauer fibration based on the choice of a generic fiber and explore its associated Brauer group. Since the Brauer group is a birational invariant for smooth projective varieties, one would want to compute the Brauer group of a Severi--Brauer fibration using only its generic fiber. This work provides an approach to do so. Additionally, since our construction does not rely on explicit polynomial equations, the techniques presented can be applied to a wide range of varieties.
	
	From now on, \( k \) is a number field that contains a primitive \(p\)th root of unity for a fixed odd prime $p$. Given an element $a \in k^\times$ and a polynomial $f(u) \in k[u]$, we describe the construction of a nice variety over $k$. Note that, there is a one-one correspondence between homogenous \( F \in k[u_0, u_1] \) whose degrees are a multiple of $p$ and $p$th power-free polynomials \( f \in k[u] \) given by
	\begin{center}
		\(
		F(u_0, u_1) \mapsto F(u, 1)
		\)
		and 
		\(
		f(u) \mapsto f\left(\frac{u_0}{u_1}\right) u_1^{p\lceil \deg f / p \rceil}.
		\)
	\end{center}
	
	\begin{construction}\label{construction}
		Consider the cyclic algebra \( (a, f) \in \Br(k(u)) \) of degree $p$. By the purity theorem \ref{purity1}, there exists an open set  U $\subset \mathbb{P}^1_k $ such that \( (a, f) \in \Br U \). Since the cyclic algebra $(a, f)$ is unramified at all closed points of $\P^{1}_{k}$ except possibly at the irreducible factors of $F$, we can take \( U = \mathbb{P}^1_k \setminus V(F) \) where $F$ is the homogeneous polynomial corresponding to $f$. Let \( X_U \) represent the family of Severi--Brauer varieties over \( U \) associated with the algebra \((a, f)\). Note that for each $q \in U$, the fiber $X_q \to \text{Spec}\, \kappa(q)$ is the Severi--Brauer variety associated with the cyclic algebra $(a, f(q)) \in \Br \kappa(q)$. Using Hironaka's theorem over \( k \), we choose a smooth projective model for \( X_U \to U \) denoted by \( \pi: X_{a,f} \to \mathbb{P}^1_k \). We like to emphasize that its generic fiber corresponds to the Severi--Brauer variety associated with \((a, f)\), which uniquely determines \( \pi: X_{a,f} \to \mathbb{P}^1_k \) up to birationality. Since a Severi--Brauer variety is geometrically integral rationally connected, we observe that $X_{a,f}$ is a nice rationally connected variety over $k$.
	\end{construction}

\begin{remark}
    The vertical Brauer group of \( X/Y \) is defined as
\[
\Br_{\text{vert}}(X/Y) = \{ A \in \Br X \mid j^*(A) \in \text{Im}[\Br(\text{Spec}\kappa(\eta)) \to \Br X_\eta] \},
\]
where $j$ denotes the natural inclusion of the generic fiber $X_{\eta} \to X.$ For families of Severi--Brauer varieties, the full Brauer group is vertical, that is, $\Br_{\text{vert}}(X)= \Br X$ \cite[Remark 11.1.12]{BG}. In particular for \( \pi: X_{a,f} \to \mathbb{P}^1_k \), $\Br (X_{a,f}) = \{ \pi^{*}(\beta) \in \Br X: \beta \in \Br k(t) \}$.
    \end{remark}

The following theorem examines some properties of the Brauer group of $X_{a,f}$.
	
	\begin{theorem}\label{p-torsion  brauer group}
		Let \( f = f_1 \dots f_r \) with \( r \geq 2 \), where each \( f_i \) is a distinct irreducible separable polynomial, and \( a \) is not a \( p \)th power in \( k[u]/\langle f_i(u) \rangle \) for any \( i \). Let \( F \in k[u_0, u_1] \) be the homogeneous polynomial corresponding to $f$. Consider the variety \( \pi: X \to \mathbb{P}^1_k \), where \( X = X_{a,f} \). We have
		\begin{enumerate}[label=\arabic*.]
			\item \label{injective} 
			$\Br k\rightarrow \Br X$ is injective. 
			\item \label{p-torsion Br}
			$\overline{\Br}X$ is $p$-torsion.
			\item \label{unramified brauer}
			Under the map \( \pi^*: \Br \kappa(\mathbb{P}^1_k) \rightarrow \Br \kappa(X) \), \( (a, f_i) \) lands in \( \Br X \) for all $i$.
			
			\item\label{nonzero Br}
			\( \pi^*(a,f_i) \) is nonzero in \( \overline{\Br}X \) for all $i$.
			
		\end{enumerate}
	\end{theorem}
    
	\begin{proof}
		$(1)$ The claim is true for conic bundles by \cite[Lemma 11.3.3]{BG}. We are providing a proof for Severi--Brauer fibration. Let $X_{\eta}$ be the generic fiber of \( \pi: X \to \mathbb{P}^1_k \). We know that the kernel of the map \( \Br k \rightarrow \Br X \hookrightarrow \Br \kappa(X) \) is equal to the kernel of the composition \( \Br k \hookrightarrow \Br k(t) \rightarrow \Br X_{\eta} \hookrightarrow \Br \kappa(X) \). Since the kernel of the map \( \Br k(t) \to \Br X_{\eta} \) is \( \langle [X_{\eta}] = [(a,f)] \rangle \)  by Theorem \ref{Amitsur} and \( \langle X_{\eta} \rangle \cap \Br k = \{0\} \), it follows that the restriction \( \Br k \to \Br X_{\eta} \) is injective. Therefore $\Br k \to \Br X$ is injective.
		
		$(2)$ This part of the proof follows a similar approach to that in \cite[Theorem 6.5.(a)]{berg}, where the case of conic bundles was addressed. We generalize their proof to Severi--Brauer fibrations arising from cyclic algebras and describe the necessary details.
		
		Let $L$ be a field extension of $k$ such that $a \in L^{\times p}$. The scheme $X_L$ is birational to \( \mathbb{P}^{p-1}_L \times \mathbb{P}^1_L \), which implies \( \Br X_L \cong \Br L \). Now let  $L=k(\sqrt[p]a)$ which is a cyclic Galois extension of \( k \) of degree \( p \) with Galois group $G$. Using the lower term exact sequence obtained using Hochschild–Serre spectral sequence we have $\overline{\Br}X \cong H^1_{et}(G, \text{Pic}\, X_L)$, which is annihilated by \(|G| = p\). This shows that \(\overline{\Br}X\) is \(p\)-torsion.

        $(3)$ Without loss of generality, fix an irreducible factor $f_i$ of $f$. We show that $\pi^{*}(a, f_i)$ lands in $\Br X$. Let $R$ be the Henselization of $k[t]$ with respect to  $f_i$ and $K$ be the fraction field of $R$. Let \( \pi_R: X_R \rightarrow \text{Spec}\,R \) be the structure morphism associated with the base change of our scheme $X$ to $\text{Spec}\,R$ with the generic fiber \( \tilde{\pi}_R: X_{\eta} \rightarrow \text{Spec}\, K \). By Amitsur's Theorem \ref{Amitsur} we have,  $\tilde{\pi}_R^{*}(a,f)=0$ in $\Br X_{\eta}$. Since $\Br \kappa(X_R)\xhookrightarrow{}\Br \kappa(X_{\eta})$, we get $\pi_R^{*}(a,f)=0$ in $\Br \kappa(X_{R})$. 
	
        Let $\hat{f_i}=f_i/f$, then we have $\pi_R^{*}(a,f_i)=\pi_R^{*}(a,f)+\pi_R^{*}(a,\hat{f_i})=\pi_R^{*}(a,\hat{f_i})$ in $\Br \kappa(X_R)$. Note that for $j\neq i$, $f_j$ is a unit in $R$ implies that \( \partial_{f_i}(a, \hat{f_i}) = a^{v_{f_i}(\hat{f_i})}=1 \). That is, \( (a, \hat{f_i}) \) lies in the kernel of the residue map \( \partial_{f_i}\), which implies that \( \pi_R^{*}(a, \hat{f_i}) \) is in \( \Br X_R \) by the purity theorem \ref{purity1}. Thus we can conclude that \( \pi_R^{*}(a, f_i) \) is also in \( \Br X_R \). Let $U=\pi^{-1}(f_i)$ and $V=\pi_R^{-1}(f_i)$, then by
        using the purity theorem \ref{purity2}, we have the following commutative diagram
		$$\begin{tikzcd}
			0 \arrow[r] & {\Br X} \arrow[rr] \arrow[d] &  & {\Br (X\setminus U)} \arrow[d] \arrow[rr] &  & {\bigoplus}_D{H^1(\kappa(D),\Q/\Z)} \arrow[d] \\
			0 \arrow[r] & {\Br X_R} \arrow[rr]         &  & {\Br (X_R\setminus V)} \arrow[rr]       &  & {\bigoplus}_{D'}{H^1(\kappa(D') ,\Q/\Z)}.          
		\end{tikzcd}$$
		
		Since the last vertical map is an isomorphism, we get $\pi^{*}(a, f_i) \in \Br X $. 
		
		$(4)$ We prove this result by contradiction. Assume \( \pi^{*}(a,f_i) = 0 \) in \( \overline{\Br}\kappa(X) \) for some $i$. Let $\pi_{\eta}:X_{\eta}\rightarrow\text{Spec}\,k(t)$ be the generic fiber of $\pi$. By assumption, we have \( \pi_{\eta}^{*}(a,f_i) = 0 \) in \( \overline{\Br}\kappa(X_{\eta}) \). By Amitsur's theorem \ref{Amitsur} we obtain the exact sequence
		\[
		1 \rightarrow \overline{\langle (a,f) \rangle} \rightarrow \overline{\Br} \kappa(\mathbb{P}^1_k) \xrightarrow{\pi_{\eta}^{*}} \overline{\Br}\kappa(X_{\eta}).
		\]
		
		Since \( \pi_{\eta}^{*}(a,f_i) = 0 \) in \( \overline{\Br}\kappa(X_{\eta})\), we have $(a,f_i)\in \text{Ker}\,\pi_{\eta}^*$. This implies $A:=[(a,f^n)]-[(a,f_i)]\in \Br k$ for some $n$. If $n$ is coprime to $p$ let $j \neq i$ be an element of $\{1, \dots, r\}$, then we have $\partial_{f_j}(A)=a^n$. As $a$ is not $p$th power in $\kappa(f_j)$, $a^n$ is not a $p$th power in $\kappa(f_j)$, thus $A\notin\Br k$. If $n$ is multiple of $p$ we consider $\partial_{f_i}(A)=a^{-1}$. Since $a$ is not $p$th power in $\kappa(f_i)$, we again have $A\notin\Br k$. In both cases, we obtain a contradiction to the implication that \( A\in \Br k \), which follows from the initial assumption. Hence our initial assumption is wrong and we conclude that $\pi^{*}(a,f_i)$ is a non-trivial element in $\overline{\Br}\kappa(X)$.
	\end{proof}
    Note that the construction of $X_{a,f}$ depends only on the generic fibre. Thus for any given central simple algebra $\alpha$ over an arbitrary number field, one gets a Severi--Brauer fibration over $\mathbb{P}^1_k$ using Hironaka's theorem, say $X_\alpha$.

\begin{corollary}
    Let $k$ be a number field. Given a cyclic algebra $\alpha$ over $k(t)$ with index $m$, the Brauer group of the associated Severi--Brauer fibration $X_\alpha$ is $m$-torsion.  
\end{corollary}
\begin{proof}
    Same as the proof of \ref{p-torsion  brauer group} $(2)$.
\end{proof}

\section{Applications}\label{applications}
\subsection{Subgroups capturing Brauer--Manin obstruction}\label{subgroups capturing BMO}

In this section, we construct a Severi--Brauer fibration \(X\) that satisfies the hypotheses of \Cref{pullback} for each odd prime \( p \). We begin by constructing a variety  as in \ref{construction} by carefully choosing a polynomial $f(x) \in k[x]$ and a constant $a \in k^{\times}$. The choice of the polynomial is guided by Lemma~\ref{required polynomial}. Then we can choose the constant depending upon the polynomial. This ensures that the variety $X_{a,f}$ has a $k$-point and a non-zero evaluation map at a non-archimedean place of $k$.

\begin{lemma}\label{required polynomial}
    For any given irreducible polynomial $g(x)\in k[x]$ of degree greater than 3, there exists a constant $c\in k$, a non-archimedean place $v$ of $k$, and a polynomial $h(x)\in k[x]$ such that the $v$-adic valuation of $g(c)$ is zero, $g(c)$ is not a $p$-th power in $k_v$, but $f(c)=g(c)h(c)$ is a $p$-th power in $k_v$, where $f(x)=g(x)h(x)$.
\end{lemma}

\begin{proof}
    Consider the curve $C$ defined by homogenisation of $y^p-g(x)$ in $\P^{2}_k$. Let $d$ denote the degree of $C$. We know that the genus of $C$ is $(d-1)(d-2)/2$ \cite[Proposition 2.15]{rm} which is greater than 1 by the hypothesis on $d$. According to Falting's theorem, $C$ has only finitely many $k$ points. In particular, there exists a $c\in k$ such that $g(c)$ is not a $p$-th power in $k$. Consequently, by the local-global principle, there is a place $v$ of $k$ such that $g(c)$ is not a $p$-th power in $k_v$ and the $v$-adic valuation of $g(c)$ is zero.

    Consider $l\in k$ such that $g(c).l$ is a $p$-th power in $k_v$ and let $h(x)$ be a polynomial in $k[x]$ such that $h(c)=l$ (we can choose this by interpolation method, for example, $h(x)=x-c+l$). Taking $f(x)=g(x)h(x)$ yields $f(c)=g(c).l$, which is a $p$-th power in $k_v$.
\end{proof}

We fix an odd prime \( p \) and consider a monic irreducible polynomial \( f_1(x) \) of degree greater than 3. By above lemma, there exist a monic polynomial \(h(x) \), a constant \( c \in k \), and a non-archimedean place \( v \) such that \( f_1(c) \) is not a \( p \)-th power in \( k_v \), but \( f(c) \) is a \( p \)-th power in \( k_v \), where \( f(x) = f_1(x)h(x) \). Let \( h(x) = f_2(x) \cdots f_n(x) \), where each \( f_j(x) \)  is an irreducible factor of \( h(x) \). Next, we choose\footnote{Any finite degree extension of $k$ contains the $p$th power of only finitely many elements of the set $\{v\ell: \ell \text{ are primes in } k\}$} an element \( a \in k^\times \) such that \( a \) is not a \( p \)-th power in \( k[u]/\langle f_i(u) \rangle \) for each \( i \in \{1, 2, \dots, n\} \), and \( v(a) = 1 \). For this choice of \( a \) and \( f \), let \(\tilde{X}\) be the variety \( X_{a,f} \) over \(\mathbb{P}^1\) constructed as in \ref{construction}.  \\
In the next proposition, we examine certain local properties of our nice rationally connected variety \( X_{a,f} \).

	\begin{proposition}\label{Main lemma}

		The variety $\tilde{\pi} : \tilde{X} \to \mathbb{P}_k^1$ constructed above satisfies the following.
		
		\begin{enumerate}[label=\arabic*.]
			\item For all Archimedean places $w$, the Severi--Brauer fibration $\tilde{X}_{k_w}$ is birational to a $\mathbb{P}^{p-1}$ bundle.
			\item There exists a non-Archimedean place $v$ of $k$ not lying over $p$ and a subgroup $H$ of $\Br\tilde{X}$ such that the map $\tilde{\phi}_{v} : \tilde{X}(k_{v}) \to \mathrm{Hom}(H, \mathbb{Q}/\mathbb{Z})$ is nonzero.
			\item For all places $\omega$, we have $0 \in \im (\tilde{\phi}_{\omega})$.	
		\end{enumerate}
		
	\end{proposition}

	\begin{proof}
		(1) The generic fiber of $\tilde{X}$ is the Severi--Brauer variety associated with the cyclic algebra $(a, f)$. After base change to an Archimedean place $w$, $a$ becomes a $p$th power, that is the algebra $(a, f)$ splits over $k_w$. Using Lemma \ref{SB variety and rational point}, we can conclude that the resulting scheme $\tilde{X}_{k_w}$ is birational to a $\mathbb{P}^{p-1}_{k_w}$ bundle. 
		
		(2) Let $v$ be the place coming from Lemma \ref{required polynomial}. Then $f(c)$ is a $p$th power in $k_v$, so the cyclic algebra \((a, f(c)) \) splits in \(\Br k_{v}\), thus the fiber \(\tilde{X}_c\) has a \(k_{v}\)-point, say \(P\). Let \( H\) be the subgroup generated by \((a,f_1)\) inside \(\Br\tilde{X} \). Consider the map $\tilde{\phi}_{v} : \tilde{X}(k_{v}) \to \mathrm{Hom}(H, \mathbb{Q}/\mathbb{Z})$ where \(\tilde{\phi}_{v}(P)\)  sends \((a, f_1)\) to \(\mathrm{inv}_{v}(a, f_1(c))\), as $P$ is above $[c:1]\in \P^1(k_v)$ and $P^{*}(a,f_1)=(a,f_1(c))$. Now, as $f_1(c)$ is not a $p$th power and the \(v\)-adic valuation of \(f_1(c)\) is non-zero, \( k_v(\sqrt[p]{f_1(c)}) \) is a degree-\(p\) unramified extension of \( k_v \), say $L$. Then the valuation of the norm group is given by \( v(N(L^{\times})) = p\mathbb{Z} \). By our choice of \( a \) we have \( v(a) = 1 \), which implies that \( a \) does not belong to the norm group. Consequently, by Lemma \ref{split algebra}, \( (a, f_1(c)) \) does not split in $k_{v}$, thus \(\mathrm{inv}_v(a, f_1(c))\) is non-zero.

		(3) Since \( f \) is monic, the variety \( \tilde{X} \) has a \( k \)-point \( Q \) over infinity which implies \( Q^*(a, f_1) = (a, 1) = 0 \). It follows that \( \tilde{\phi}_w(Q) = 0 \) for all places $\omega$.
	\end{proof}

    In the following theorem, we construct a Severi--Brauer fibration that
    violates the Hasse principle, without resorting to explicit calculations,
    by pulling back \(\tilde{X}\) along a polynomial \(g\). We will prove Theorem~\ref{main}, as a particular case of the following theorem	.
    
	\begin{theorem} \label{SB fibration violating Hasse}
		There exists a smooth projective geometrically integral rationally connected variety $X$ over $k$ such that
		\begin{enumerate}[label=\arabic*.]
			\item  \( \overline{\Br}X \) is a non-trivial $p$-torsion group,
			\item $X$ has $k_{\nu}$ points for all places $\nu$ of $k$,
			\item The Brauer--Manin set \( X(\mathbb{A}_k)^B \) is empty, where \( B \) is a subgroup of \( \overline{\Br}X \).
		\end{enumerate}
	\end{theorem}
	
	\begin{proof}

   This proof follows from Theorem~\ref{pullback}. Let \(\Tilde{X}\) over 
\(\mathbb{P}^1_k\) be the variety considered in Proposition~\ref{Main lemma}, 
and set \(\alpha = \tilde{\pi}^*(a,f_1)\). Proposition~\ref{Main lemma} shows 
that both \(\Tilde{X}\) and \(\alpha\) satisfy the hypotheses of Theorem~\ref{pullback}. Let \(X \to \mathbb{P}^1\) be the base change of 
\(\pi : \Tilde{X} \to \mathbb{P}^1_k\) by the polynomial 
\(g \in k[t]\) provided by the theorem. Define \(B\) to be the subgroup of 
\(\Br X\) generated by the image of \(\alpha\), that is
\( B = \langle \pi^*(a, g \circ f_1) \rangle .\) Parts~(2) and~(3) of the theorem follow.
Moreover, the polynomial \(g\) can be chosen so that each 
\(f_i \circ g\) is irreducible and separable over \(k(\sqrt[p]{a})\),
which ensures that \(a\) is not a \(p\)th power in \(k[u]/(f_i \circ g)\). 
Hence, by Theorem~\ref{p-torsion brauer group}, we conclude that 
\(\overline{\Br}X\) is a nonzero \(p\)-torsion group. This establishes part~(1).

\end{proof}

\begin{proof}[{Proof of Corollary \ref{sub of main}}]
	Consider the nice $k$-variety $X$  constructed in Theorem~\ref{SB fibration violating Hasse}. The fact that \( \overline{\Br}X \) is \( p \)-torsion implies that all of its non-trivial elements have order $p$. In particular, the subgroup $\langle (a, f_1 \circ g) \rangle$, which captures the obstruction is isomorphic to $\Z/p\Z$. We conclude that $\Z/p\Z$ can capture Brauer--Manin obstruction for any odd prime $p$. 
\end{proof}

\subsection{Varieties with index one failing Hasse principle}\label{index one failing Hasse}
In this subsection we will apply the pullback method to construct a variety which is everywhere locally
soluble and has index one but no global point.

\begin{lemma}\label{unit and splitting}
Let $(a,b)$ denote a cyclic algebra of degree $n$ over a number field $k$ containing $\zeta_n$. 
If for a place $v$ of $k$ both $a$ and $b$ are units in $\mathcal{O}_v$, then the associated Severi--Brauer variety has a $k_v$-point.
\end{lemma}

\begin{proof}
This result is standard, but we provide a proof for the reader’s convenience. Let $\mathcal{O}_v$ be the valuation ring at $v$, which is a discrete valuation ring with residue field $\kappa$. 
We have the following exact sequence of Brauer groups
\[
0 \to \Br\mathcal{O}_v \to \Br k_v 
\xrightarrow{\Res} H^1(\kappa,\Q/\Z).
\]
The residue map can be described explicitly as
$$ \Res(a,b) = 
(-1)^{v(a)v(b)} \cdot \frac{a^{v(b)}}{b^{v(a)}} 
\in \kappa^\times / \kappa^{\times n} 
\subset H^1(\kappa,\Q/\Z).$$
If $a,b \in \mathcal{O}_v^\times$, then $v(a)=v(b)=0$, hence $\Res(a,b)=0$. 
It follows that $(a,b) \in \Br \mathcal{O}_v$. 
Since $\Br \mathcal{O}_v=0$\cite[Corollary 6.9.3.]{poonen}, we conclude that $(a,b)$ is split over $k_v$. Thus the corresponding Severi--Brauer variety has a $k_v$-point.
\end{proof}

\begin{lemma}\label{closed point of certain degree}
Let $f = \prod_{i=1}^r f_i$ be the factorization of $f$ into irreducible polynomials over a number field $k$ containing the $p$th roots of unity. 
Then the Severi--Brauer fibration $X_{a,f}$ has closed points of degree $\deg(f_i)$.
\end{lemma}

\begin{proof}
By \cite[Theorem 74]{kol}, the Severi--Brauer variety associated to the cyclic algebra $(a,f(t))$ is birational to the norm hypersurface
\[
N_{k(\sqrt[p]{a})(t)/k(t)}(\vec{z}) = f(t),
\]
where $N_{k(\sqrt[p]{a})(t)/k(t)}$ denotes the norm from $k(\sqrt[p]{a})(t)$ to $k(t)$. The fibration $X_{a,f}$, obtained by spreading out the Severi--Brauer variety over $\mathbb{P}^1_k$, is thus birational to the corresponding spreading out of this norm hypersurface, which we denote by $N_{a,f}$. At points where $f(t)=0$, the fibers of both $X_{a,f}$ and $N_{a,f}$ are degenerate. 
Let $\theta_i$ be a root of $f_i$. Then over $t=\theta_i$, the variety $N_{a,f}$ admits a $\kappa(\theta_i)$-rational point, namely $(\vec{z}=0,\; t=\theta_i).$
This gives a closed point of degree $\deg(f_i)$.  Since $X_{a,f}$ and $N_{a,f}$ are birational, it follows that $X_{a,f}$ also possesses a closed point of degree $\deg(f_i)$.
\end{proof}
We are in a position to prove the second main theorem of the paper.
\begin{proof}[{Proof of Theorem \ref{main}}]
Let $a=2$ and $f(t)=(t^p+\zeta_p)(t^p-\zeta_p^{-1}).$
Consider the cyclic algebras $A_1=(2,t^p+\zeta_p), A_2=(2,t^p+\zeta_p^{-1}).$ Since $2$ is not a $p$th power in $k(\sqrt[p]{\zeta_p})$, it follows from Theorem~\ref{p-torsion  brauer group} that $A_1$ and $A_2$ define nontrivial elements of $\overline{\Br}X_{2,f}$ and represent the same class in it. We note that $2$ is unramified in $\Q(\zeta_p)$. Let $m$ be the smallest integer such that $2^m\equiv1\pmod{p}$.
We now show that if $v$ is place not dividing $2$ then for any $P_v \in X_{2,f}(k_v)$ we have $\inv_v A(P_v)=0$, where $A$ denotes either $A_1$ or $A_2$.

    \emph{Archimedean places.}  
    At an archimedean place $v$, the element $2$ is a $p$th power in $k_v$. 
    Hence for $P_v \in X_{2,f}(k_v)$ with $\pi(P_v)=[x:1]$,
    $ P_v^*(2,t^p+\zeta_p) = (2,x^p+\zeta_p) = 0 \in \Br k_v,$ so $\inv_v A(P_v)=0$.
    
    \emph{Non-archimedean places not dividing $2$.}  
    Let $v$ be such a place and let $P_v \in X_{2,f}(k_v)$ with $\pi(P_v)=[x:1]$.
    
        If $v(x)<0$, then $v(x^p+\zeta_p)=p\cdot v(x),$ hence $x^p+\zeta_p = u \cdot \ell^p$ for some \(\ell\) with $u \in \mathcal{O}_v^\times$. 
        Since $2$ is also a unit at such $v$, Lemma~\ref{unit and splitting} gives $(2,x^p+\zeta_p)=(2,u)=0 \in \Br k_v.$

         If $v(x)\geq 0$, then at least one of $x^p+\zeta_p$ or $x^p-\zeta_p^{-1}$ must be a unit. Otherwise, both would vanish modulo maximal ideal $\mathfrak{m}_v$, implying $\zeta_p+\zeta_p^{-1}\equiv 0 \pmod{\mathfrak{m}_v}$, so $\zeta_p^2\equiv -1 \pmod{\mathfrak{m}_v}.$
        Consequently $\zeta_p^4 \equiv 1 \pmod{\mathfrak{m}_v}$. 
        Since $\zeta_p^p \equiv 1 \pmod{\mathfrak{m}_v}$ and $(4,p)=1$, we deduce 
        $\zeta_p \equiv 1 \pmod{\mathfrak{m}_v}$, which further implies 
        $\zeta_p^{-1}\equiv 1 \pmod{\mathfrak{m}_v}$. Thus $2\equiv0 \pmod{\mathfrak{m}_v}$, 
        contradicting the assumption that $v$ is not above $2$. 
        Hence one of $x^p+\zeta_p$ or $x^p-\zeta_p^{-1}$ is a unit, and by choosing the appropriate 
        representative ($A_1$ or $A_2$), we obtain $(2,u)=0 \in \Br k_v.$ Therefore if $v$ is place not dividing $2$ then for any $P_v \in X_{2,f}(k_v)$, $\inv_v A(P_v)=0$,

\emph{Non-archimedean places dividing $2$}. We claim that there exists $P_v \in X_{2,f}(k_v)$ such that $\operatorname{inv}_v A(P_v)\neq 0$. By our choice of model of $X_{2,f}$, over the point $[0:1]$ the fiber is the Severi--Brauer variety corresponding to $(2,-1)=0 \in \Br k$, and hence has a $k$-rational point $P$.  Let $P_v$ denote its image in $k_v$. Then $P_v^*(2,t^p+\zeta_p)=(2,\zeta_p)$. Since $\zeta_p$ is not a $p$th power in $k_v$, $L_v=k_v(\sqrt[p]{\zeta_p})$ is a degree-$p$ unramified extension of $k_v$.
This implies $v(N_{L_v/k_v})=p\Z$, but $v(2)=1$, thus we conclude $2$ is not a norm and hence $\inv_v A(P_v)\neq 0$.

We apply the pullback method to the Brauer class $A_1$, taking  $v_1$ to be a prime above $2\in\Omega_k$ as in \Cref{SB fibration violating Hasse}. 
By \cite[Lemma~4.3]{berg}, we have $q = 2^m$ in this case. 
Therefore, the degree of the polynomial $g$ used in the pullback must be even, and so is $\deg(f \circ g)$. It then follows from Lemma~\ref{closed point of certain degree} that the variety $X_{2, f \circ g}$ admits a closed point of even degree. Since the good fibers are Severi--Brauer varieties of degree $p$, they have closed points of degree $p$. Combining these, we deduce that the variety $X_{2, f \circ g}$ has index one, and by proof of Corollary~\ref{sub of main}, $X_{2, f \circ g}$ has Brauer--Manin obstruction to the existence of rational points given by a single Brauer class of order $p$.
\end{proof}


\textbf{Statements and declarations:} The authors have no competing interests to declare that are relevant to the content of this article.

\vspace{2mm}

\textbf{Acknowledgements:}
    M. Biswas was partially supported by PMRF (ID: 0703045) and UC Doctoral Scholarship (UCID: 96651612). D.C.R was partially supported by DST-INSPIRE Fellowship (Reg No: IF210208). B. Samanta was partially supported by PhD Fellowship (File No: 09/936(0315)/2021-EMR-I) of CSIR, India. The authors thank the organisers of the `Instructional Workshop on Rational Points,' funded by NWO through an ENW-XL grant, held at Groningen for providing a platform for valuable discussions. B. Samanta is thankful to Infosys travel grant for supporting his participation in the workshop. The authors wish to express their gratitude to Amit Hogadi for his continuous support and guidance throughout this project. The authors are grateful to Anthony V\'{a}rilly-Alvarado, Brendan Creutz, and Bianca Viray for their valuable comments and suggestions. The authors would like to thank the anonymous referee for their helpful comments and suggestions, which significantly improved the exposition of the manuscript.

\end{document}